\documentclass{amsart}
\usepackage[utf8]{inputenc}
\usepackage{amsmath}
\usepackage{amsthm}
\usepackage{amsfonts}
\usepackage[english]{babel}
\usepackage{mathtools}

\theoremstyle{definition}

\newcommand{\N}{\mathbb{N}}

\newcommand{\R}{\mathbb{R}}
\newcommand{\HH}{\mathbb{H}}

\newtheorem{theorem}{Theorem}[section]
\newtheorem{remark}[theorem]{Remark}
\newtheorem{definition}[theorem]{Definition}
\newtheorem{proposition}[theorem]{Proposition}
\newtheorem{corollary}[theorem]{Corollary}

\newtheorem{lemma}[theorem]{Lemma}

\newcommand{\s}{{\sigma}}
\newcommand{\dd}{{\delta}}

\title{Fixed points and orbits in skew polynomial rings}
\author{Adam Chapman}
\email{adam1chapman@yahoo.com}
\address{School of Computer Science, Academic College of Tel-Aviv-Yaffo, Rabenu Yeruham St., P.O.B 8401 Yaffo, 6818211, Israel}
\author{Elad Paran}
\email{paran@openu.ac.il}
\address{Department of Mathematics and Computer Science, The Open University of Israel, 1 University Road, P. O. Box 808, Ra'anana 43107}

\keywords{Skew polynomials; Division Rings; Arithmetic Dynamics; Periodic Points; Noncommutative Algebra}  
\subjclass[2010]{primary 16S36; secondary 16K20, 37P35, 37C25}

\begin{document}

\maketitle

\begin{abstract} We study orbits and fixed points of polynomials in a general skew polynomial ring $D[x,\sigma, \delta]$. We extend results of the first author and Vishkautsan on polynomial dynamics in $D[x]$. In particular, we show that if $a \in D$ and $f \in D[x,\sigma,\delta]$ satisfy $f(a) = a$, then $f^{\circ n}(a) = a$ for every formal power of $f$. More generally, we give a sufficient condition for a point $a$ to be $r$-periodic with respect to a polynomial $f$. Our proofs build upon foundational results on skew polynomial rings due to Lam and Leroy. \end{abstract}

\section{Introduction}

The recent work \cite{CV21} studies discrete dynamics in the ring $D[x]$ of polynomials in a central variable $x$ over a division ring $D$. For a polynomial $f(x) = \sum f_i x^i \in D[x]$, one defines its formal powers $f^{\circ n}$ inductively, by $f^{\circ 0} = x$ and $f^{\circ (n+1)}(x) = \sum f_i (f^{\circ n})^i$. A point $a \in D$ is called $r$-periodic with respect to $f$, if $f^{\circ nr}(a) = a$ for all $n \geq 0$. It is demonstrated in \cite[Example 4.1]{CV21} that $f^{\circ r}(a) = a$ does not ensure that $a$ is $r$-periodic with respect to $f$. However, a fixed points theorem is given in \cite[Corollary 4.7]{CV21}: For $r = 1$, if $f(a) = a$ then $f^{\circ n}(a) = a$ holds, for all $n \geq 0$. 

The purpose of this note is to extend and generalize the results of \cite{CV21} for a general skew polynomial ring $R = D[x,\sigma, \delta]$, where $\sigma$ is an endomorphism of $D$ and $\delta$ is a $\sigma$-derivation. These rings were first systematically investigated by Ore in \cite{Ore33}, following an initial work by Noether and Schmeidler in \cite{NS20}. There have since been many works studying these rings, e.g. \cite{Goo92}, \cite{GL94a}, \cite{SZ02}, \cite{LaL04}, \cite{Tae06}, \cite{LaL08}, \cite{BU09}, \cite{MK19}, \cite{AP20b}. The paper \cite{LL88} of Lam and Leroy, in particular, will be instrumental in the present work. 

Our first result extends the fixed points theorem of \cite{CV21} to $D[x,\sigma, \delta]$. We extend the formal composition operation $f \mapsto f^{\circ n}$ to the context of skew polynomials rings, and prove that if $f \in $ $D[x,\sigma, \delta]$ satisfies $f(a) = a$ then $f^{\circ n}(a) = a$ for all $n \geq 0$, see Theorem \ref{thm:fixed}.

Next, we consider general orbits of points in $D$ with respect to polynomials in $D[x,\sigma,\delta]$. Our main result is a sufficient condition for a point $a \in D$ satisfying $f^{\circ r}(a) = a$ to be $r$-periodic, see Theorem \ref{periodic}. This result generalizes \cite[Theorem 4.6]{CV21} from $D[x]$ to $D[x,\sigma,\delta]$.

The paper is organized as follows: In \S\ref{sec:prelim} we briefly review some basic material concerning skew polynomial rings; In \S\ref{sec:fixed} we prove the fixed points theorem for $D[x,\sigma,\delta]$. In the case of $D[x]$ (where $\sigma = \operatorname{Id}$ and $\delta =0$) our proof reduces to a simpler and shorter proof than the one given in \cite[Corollary 4.7]{CV21}. Finally, in \S\ref{sec:orbit} we prove Theorem \ref{periodic} concerning periodic points.

\section{Preliminaries}\label{sec:prelim}

Throughout this paper, let $R = D[x,\s,\dd]$ be a skew polynomial ring, where $\s$ is an endomorphism of the division ring $D$ and $\dd$ is a $\s$-derivation on $D$. That is, $\dd$ is an additive map that satisfies the generalized Leibniz rule $\delta(ab) = \sigma(a) \delta(b) + \delta(a)b$, see \cite[p.~309]{LL88}\footnote{This rule ensures that $\deg(pq) = \deg(p)+\deg(q)$ for all $p,q \in R$, see \cite[p.~401]{Ore33}.}. 

A remark concerning notation: Due to space considerations versus readability considerations, we will sometimes use the notation $a^\s$ and sometimes $\s(a)$ to denote the image of $a \in D$ under $\s$. Similarly, we shall often write $a'$ for $\delta(a)$ when the derivation is clear from the context. Also, if $b$ is a non-zero element in $D$, we will denote by $a^b$ the $(\sigma,\delta)$-conjugation $a^b = b^\sigma a b^{-1}+b'b^{-1}$, see \cite[p.~311]{LL88}.

For a polynomial $f = \sum f_i x^i \in R$ and an element $a \in D$, the substitution $f(a)$ is defined as follows: First, for an integer $n \geq 0$, one defines the generalized power $a^{[n]}$ inductively, by $a^{[0]} = 1$ and $a^{[n+1]} = \s(a^{[n]})a+\dd(a^{[n]})$ for all $n \geq 0$ (see \cite[p.~310]{LL88}, where $a^{[n]}$ is denoted by $N_n(a)$). Then $f(a)$ is defined to be $\sum f_i a^{[i]}$. From this definition it follows that $f(a)$ is the unique element in $D$ for which $\big(f-f(a)\big)$ is a left multiple of $x-a$ in the ring $R$, see \cite[Lemma 2.4]{LL88}.

Next, let us trivially extend $\s$ and $\dd$ from $D$ to $R$ by $\s(\sum f_i x^i) = \sum f_i^\s x^i$ and $\dd(\sum f_i x^i) = \sum \dd(f_i) x^i$. We define the powers $f^{[n]}$ of $f \in R$ as we did for $D$. That is, $f^{[0]} = 1$ and $f^{[n+1]} = \s(f^{[n]})f+\dd(f^{[n]})$ for all $n \geq 0$. The formal composition $f \circ g$ will be defined as $\sum f_i g^{[i]}$. Note that this definition is consistent with the substitution of elements in $D$. That is, if $a 
\in D$, then viewing $a$ as a (constant) polynomial in $R$, we have $f 
\circ a = f(a)$. Note also that the composition generalizes the one studied in \cite{CV21} for the ring $D[x]$, since in this case we have $f^{[n]} = f^n$ for all $f \in D[x]$ and $n \geq 0$. Finally, we define the formal power $f^{\circ n}$ inductively, by $f^{\circ 0} = x$ and $f^{\circ (n+1)} = f \circ f^{\circ n}$ for all $n \geq 0$, as in \cite{CV21}. 

We note that the extended $\sigma$ and $\delta$ need not be an endomorphism and a $\sigma$-derivation of $R$, respectively. However, they are if one assumes that $\sigma$ and $\dd$ commute:

\begin{proposition}\label{extend1} If $\s$ and $\dd$ commute on $D$, then $\sigma$ is an endomorphism of $R$. \end{proposition} \begin{proof} Clearly $\sigma$ is additive, and for any $a \in D$ and $f \in R$ we have $\s(af) = \s(a)\s(f)$. Thus it remains to prove that $\s(x^k a x^m) = x^k a^\sigma x^m$ for all $a\in D$ and $k,m \geq 0$. For $k = 0$ there is nothing to prove. Suppose by induction that we have proven the claim for a given $k$ (for all $a$ and all $m$), then$$\s(x^{k+1} a x^m) = \s(x^k (a^\sigma x + \delta(a)) x^m) = \s(x^k (a^\sigma x) x^m+x^k (\delta(a)) x^m) = $$
$$\s(x^k (a^\sigma ) x^{m+1})+\s(x^k (\delta(a)) x^m) = x^k a^{\s^2}x^{m+1}+x^k \delta(a)^\s x^m = $$

$$x^k (a^{\s^2}x+\delta(a)^\s)  x^m = x^k (a^{\s^2}x+\delta(a^\s))  x^m = x^k (xa^\s)  x^m =x^{k+1}a^\s  x^m, $$ as needed. \end{proof}

\begin{proposition}\label{extend2} If $\dd$ and $\s$ commute on $D$, then $\delta$ is a $\s$-derivation on $R$. \end{proposition} \begin{proof} Clearly $\dd$ is additive. We must show that for all $f,g \in R$ we have $\dd(fg) = f^\sigma \dd(g)+ \dd(f)g$. If $f$ is in $D$ then this is obvious, and so by additivity we may assume that $f = x^k, g = a x^m$ for some $a \in D$ and $k,m \geq 0$. For $k = 0$ the claim is clear. Assume that we have proven the claim for all $m$ and $a$ and for a fixed $k$. Then for $k+1$ we have $$\dd(x^{k+1}ax^m) =\dd(x^k (a^\sigma x + \delta(a)) x^m) = \dd(x^k (a^\sigma x) x^m+x^k (\delta(a)) x^m) = $$$$\dd(x^k (a^\sigma  x^{m+1}))+\dd(x^k (\delta(a) x^m)) =$$
$$ x^k\dd(a^\sigma)x^{m+1}+x^k a^\sigma x^{m+1}+x^k\dd^2(a)x^m + x^k\dd(a)x^m= $$
$$ x^k(\dd(a^\sigma)x +\dd^2(a))x^m +x^k(a^\sigma x+\dd(a)) x^m = $$
$$ x^k(x\dd(a))x^m +x^k(xa) x^m = x^{k+1}\dd(a)x^m+x^{k+1}ax^m =$$

$$  (x^{k+1})^\sigma \dd(ax^m)+\dd(x^{k+1}) (ax^m), $$ as needed. \end{proof}

Another property that holds if one assumes that $\delta$ and $\sigma$ commute is the following:

\begin{proposition}\label{power_commute} If $\dd$ and $\s$ commute on $D$, then for all $a \in D$ and $n \geq 1$ we have $\s(a^{[n]}) = (\s(a))^{[n]}$. \end{proposition} \begin{proof} For $n = 1$ there is nothing to prove. Suppose we have proven the claim for $n \geq 1$. Then for $n+1$, using the induction hypothesis and the fact that $\delta$ and $\s$ commute, we have:

$$ (a^\s)^{[n+1]} = ((a^\s)^{[n]})^\sigma \cdot a^\sigma +\delta((a^\s)^{[n]}) = \s((a^{[n]})^\s a)+\dd((a^{[n]})^\s) = $$
$$\s((a^{[n]})^\s\cdot a)+ \s(\dd(a^{[n]})) = \s\big((a^{[n]})^\s\cdot a + \dd(a^{[n]})\big) = \s(a^{[n+1]}),$$ as needed.\end{proof}

Propositions \ref{extend1}, \ref{extend2} and \ref{power_commute} will not be needed in the sequel, yet they may be of independent interest. Note that in the unmixed cases where $R = D[x,\sigma]$ or $R = D[x,\delta]$ (that is, when $\delta = 0$ or $\sigma = \operatorname{Id}$, respectively) $\sigma$ and $\delta$ trivially commute, and so in all such cases the assertions of the above propositions hold.

\section{Fixed points}\label{sec:fixed}

In this section we generalize \cite[Corollary 4.7]{CV21}. 

\begin{lemma}\label{good} Suppose that $a \in D$ and $f \in R$ satisfy $f(a) = a$. Then $f^{[n]}(a) = a^{[n]}$ for all $n \geq 1$. \end{lemma} \begin{proof} For $n = 1$ the claim holds trivially. Suppose by induction that we have proven the claim for $n$. Then for $n+1$, we have $f^{[n+1]} = \s(f^{[n]})f+\dd(f^{[n]})$. Write $f^{[n]} = \sum g_i x^i$. 

Let us first compute $\big(\s(f^{[n]})f\big)(a)$. If $f(a) = a \neq 0$, then by the product formula \cite[Theorem 2.7]{LL88}, we have $$\big(\s(f^{[n]})f\big)(a) = \big(\s(f^{[n]})(a^a)\big)\cdot a = \sum g_i^\s(a^a)^{[i]}a.$$  
By \cite[Proposition 2.9(3)]{LL88}, we can write the last expression as $\sum g_i^\s a^{[i+1]}$, which equals $$\sum g_i^\s (\s(a^{[i]}))a+\dd(a^{[i]}) = (\sum g_i^\s \s(a^{[i]}))a+\sum g_i^\s\dd(a^{[i]}) = $$ $$(f^{[n]}(a))^\s a + \sum g_i^\s\dd(a^{[i]}) = (a^{[n]})^\s a + \sum g_i^\s\dd(a^{[i]}).$$ 
Thus $$f^{[n+1]}(a) = \big(\s(f^{[n]})f\big)(a)+\dd(f^{[n]})(a) = (a^{[n]})^\s a + \sum g_i^\s\dd(a^{[i]})+\dd(f^{[n]})(a).$$ Note that this equality also holds in the case where $a = 0$, since in this case we have by induction that $a^{[k]} = 0$ for all $k \geq 0$, and in particular $g_0 = f^{[n]}(0) = 0$ by the induction hypothesis. 

Next, by applying the $\sigma$-derivation $\delta$ to the equality $a^{[n]} = f^{[n]}(a) =\sum g_i a^{[i]}$ we get, using the Leibniz rule, that $$\dd(a^{[n]}) = \sum g_i^\s\delta(a^{[i]}) + \sum \delta(g_i)a^{[i]} = \sum g_i^\s\delta(a^{[i]}) + \delta(f^{[n]})(a).$$ Plugging this equality into the expression for $f^{[n+1]}(a)$ established above, we get that $f^{[n+1]}(a) = (a^{[n]})^\s a + \dd(a^{[n]}) = a^{[n+1]}$. \end{proof}

We shall say that a point $a \in D$ is a {\bf fixed point} of $f \in D[x,\sigma,\dd]$, if $f^{\circ n}(a) = a$ for all $n \geq 1$.

\begin{theorem}\label{thm:fixed} Let $f \in D[x,\sigma,\dd]$ and $a \in D$. If $f(a) = a$ then $a$ is a fixed point of $f$. \end{theorem} \begin{proof} For $n = 1$ there is nothing to prove. Suppose by induction that $f^{\circ n}(a) = a$ for a given $n \geq 1$. Write $g = f^{
\circ n}$. By Lemma \ref{good} we have $g^{[i]}(a) = a^{[i]}$ for all $i \geq 0$. Write $f = \sum f_i x^i$. Then $f^{\circ (n+1)}(a) = (\sum f_i g^{[i]})(a) =  \sum f_i (g^{[i]}(a)) = \sum g_i a^{[i]} = g(a) = a$. \end{proof}

\section{Orbits and periodic points}\label{sec:orbit}


\begin{definition} We say that an element $a \in D$ is {\bf stable} with respect to $(\sigma,\delta)$ if $\sigma\big((a^{[n]})'\big) = \big(\sigma(a^{[n]})\big)'$ for all $n \in \N$. \end{definition}

When $\sigma$ and $\delta$ are clear from the context, we will simply say that an element $a$ is {\bf stable}, rather than stable with respect to $(\sigma,\delta)$.

Note that for the condition $\sigma\big((a^{[n]})'\big) = \big(\sigma(a^{[n]})\big)'$ to hold for all $n \in \N$ it is not generally enough to assume that $\sigma(a)' = \sigma(a')$. However, in the unmixed cases of $D[x,\sigma]$ (that is, $\delta = 0$) or $D[x,\delta]$ (that is, $\sigma = \operatorname{Id}$), every element is stable. More generally, if $\sigma$ and $\delta$ commute, then every element in $D$ is stable. This holds in natural mixed cases, for example when $D$ is a rational function field $K(t)$ over a (commutative) field $K$, $\sigma$ is the $K$-automorphism $t \mapsto t+1$, and $\delta$ is the standard derivation with respect to $t$. 



\begin{definition} Let $a,b \in D$. We say that $a$ {\bf interchanges} with $b$ if either $a = b$, or   $a^\sigma b+a' = b^\sigma a+b'$ and $a$ is stable. \end{definition}

Note that interchangeability is by definition a reflexive condition, but not generally symmetric.
However, if every element in $D$ is stable, as in the common cases mentioned above, interchangeability is clearly a symmetric relation. 

Note also that in the case of $D[x]$ (that is, when $\delta = 0$ and $\sigma = \operatorname{Id}$), $a$ interchanges with $b$ if and only if $a$ and $b$ commute.

\begin{lemma}\label{zero_case} If $a' = 0$ and $a$ is stable then $\big(a^{[n]}\big)' =  0$ for all $n \geq 0$. \end{lemma} \begin{proof} For $n = 0$ we have $(a^{[0]})' = 1' = 0$. Suppose we have proven the claim for a given $n$. Then for $n+1$ we have, by the Leibniz rule: $$\big(a^{[n+1]}\big)' = \big(\sigma(a^{[n]})a+(a^{[n]})'\big)' = \sigma^2(a^{[n]})a'+\big(\sigma(a^{[n]})'\big)a+(a^{[n]})''.$$ By the induction hypothesis, and since $a$ is stable, this expression equals 
$$\sigma^2(a^{[n]})0+\sigma\big((a^{[n]})'\big)a+0' = 0+\sigma(0)a+0' = 0.$$\end{proof}

The key proposition is the following one:

\begin{proposition}\label{phew} If $a$ interchanges with $b$ and $b \neq 0$ then $\big(a^b\big)^{[n]}b = \sigma\big(a^{[n]}\big)b +(a^{[n]})'$ for all $n \geq 0$. \end{proposition} \begin{proof} For $n = 0$ both sides of the claimed equality are $b$. 

Suppose we have proven the claim for a given $n$. Then for $n+1$ we have, using the induction hypothesis:
$$\big(a^{b}\big)^{[n+1]}b = \big( \sigma((a^b)^{[n]})a^b+((a^b)^{[n]})'\big)b=$$
$$ \sigma\big( \sigma(a^{[n]})+(a^{[n]})'b^{-1}\big)a^bb+\big(\sigma(a^{[n]})+(a^{[n]})'b^{-1})\big)'b = $$
$$\sigma^2\big(a^{[n]}\big)a^bb+\sigma\big((a^{[n]})'b^{-1}\big)a^bb+\big(\sigma(a^{[n]})\big)'b+\big((a^{[n]})'b^{-1}\big)'b.$$ In this expression, replace the first occurrence of $a^bb$ with $a^\sigma b+a'$, the second occurrence with $b^\sigma a+b'$, and apply the Leibniz rule to the last summand, to obtain the expression:

$$\sigma^2\big(a^{[n]}\big)(a^\sigma b+a')+\sigma\big((a^{[n]})'b^{-1}\big)(b^\sigma a+b')+\big(\sigma(a^{[n]})\big)'b+ \sigma\big((a^{[n]})'\big)(b^{-1})'b+  (a^{[n]})'',$$ which after expanding and rearranging terms becomes:

$$\sigma^2(a^{[n]})a^\sigma b+\big(\sigma(a^{[n]})\big)'b+\sigma^2(a^{[n]})a'+\sigma\big((a^{[n]})'\big)a+\sigma\big((a^{[n]})'b^{-1}\big)b'+ \sigma\big((a^{[n]})'\big)(b^{-1})'b+  (a^{[n]})'' $$$$ = \sigma^2(a^{[n]})a^\sigma b+\big(\sigma(a^{[n]})\big)'b+\sigma^2(a^{[n]})a'+\sigma\big((a^{[n]})'\big)a+(a^{[n]})''+\sigma\big((a^{[n]})'\big)\big(\sigma(b^{-1})b'+(b^{-1})'b\big).$$ Note that by the Leibniz rule, $\big(\sigma(b^{-1})b'+(b^{-1})'b\big) = (b^{-1}b)' = 1' = 0$, and so we may omit the last summand from the presented expression to get
$$\sigma^2(a^{[n]})a^\sigma b+\big(\sigma(a^{[n]})\big)'b+\sigma^2(a^{[n]})a'+\sigma\big((a^{[n]})'\big)a+(a^{[n]})''.$$ Now, let us add and subtract $\sigma\big((a^{[n]})'\big)b$ from this expression to get:
$$\big(\sigma^2(a^{[n]})a^\sigma+\sigma\big((a^{[n]})'\big)\big) b-\sigma\big((a^{[n]})'\big)b+\big(\sigma(a^{[n]})\big)'b+\sigma^2(a^{[n]})a'+\sigma\big((a^{[n]})'\big)a+(a^{[n]})'' = $$$$\sigma(a^{[n+1]}) b+\big(\big(\sigma(a^{[n]})\big)'-\sigma\big((a^{[n]})'\big)\big)b+\sigma\big((a^{[n]})'\big)a+\sigma^2(a^{[n]})a'+(a^{[n]})''. $$
Next, we add and subtract $\big(\sigma(a^{[n]}))'a$ to this expression to get:
$$\sigma(a^{[n+1]}) b+\big(\big(\sigma(a^{[n]})\big)'-\sigma\big((a^{[n]})'\big)\big)b+\big(\sigma\big((a^{[n]})'\big)-\big(\sigma(a^{[n]}))'\big)a+\big(\sigma(a^{[n]}))'a+\sigma^2(a^{[n]})a'+(a^{[n]})'' $$
$$=\sigma(a^{[n+1]}) b+\big(\big(\sigma(a^{[n]})\big)'-\sigma\big((a^{[n]})'\big)\big)(b-a)+\big(\sigma(a^{[n]}))'a+\sigma^2(a^{[n]})a'+(a^{[n]})'' .$$ Our assumption that $a$ interchanges with $b$ implies that either $b - a = 0$ or $\big(\sigma(a^{[n]})\big)'-\sigma\big((a^{[n]})'\big) = 0$, hence the presented expression reduces to

$$\sigma(a^{[n+1]}) b+\sigma^2(a^{[n]})a'+\big(\sigma(a^{[n]}))'a+(a^{[n]})''.$$ Which by the Leibniz rule equals
$$\sigma(a^{[n+1]}) b+\big(\sigma(a^{[n]}) a\big)'+(a^{[n]})'' = \sigma(a^{[n+1]}) b+\big(\sigma(a^{[n]}) a+ (a^{[n]})'\big)' = $$
$$\sigma(a^{[n+1]}) b+\big(a^{[n+1]}\big)' $$
As needed. \end{proof}

\begin{corollary}\label{cor1} Let $f \in D[x,\sigma,\delta]$. If $a$ interchanges with $f(a)$ then $f^{[n]}(a) = \big(f(a)\big)^{[n]}$ for all $n \geq 1$. \end{corollary}\begin{proof} Let $b = f(a)$. For $n = 1$ there is nothing to prove. Suppose we have proven the claim for a given $n$. Write $g = f^{[n]}$ as $\sum g_i x^i$. 

Suppose first that $b = 0$. Then $$f^{[n+1]}(a) = (\sigma(g)f+\delta(g))(a) = 0+\delta(g)(a) = \sum g_i'a^{[i]}. $$Since $a$ interchanges with $b$, we have $a =0$ or $a' = 0$ and $a$ is stable. In both cases we have $\big(a^{[i]}\big)' = 0$ for all $i \geq 0$. Indeed, in the first case this holds trivially, and in the second case by Lemma \ref{zero_case}. Thus we may write:
$$f^{[n+1]}(a) = \sum g_i'a^{[i]} = \sum g_i^\sigma (a^{[i]})'+g_i'a^{[i]}.$$By the Leibniz rule we may write the right-hand side of the last equality as $(\sum g_i a^{[i]})' = (g(a))'$. By the induction hypothesis, $g(a) = 0^{[n]} = 0$, hence $f^{[n+1]}(a) = 0' = 0$. Thus $f(a)^{[n+1]} = 0^{[n+1]} = 0 =  f(a)^{[n+1]}$. 

Next, suppose that $b \neq 0$. For $n+1$ we have $$f^{[n+1]}(a) = (\sigma(g)f+\delta(g))(a) = \big(\sigma(g)(a^{b})\big)b+\delta(g)(a)=$$ $$\sum g_i^\sigma (a^b)^{[i]}b + \sum g_i' a^{[i]} = \sum \big(g_i^\sigma (a^b)^{[i]}b + g_i' a^{[i]}\big). $$
By Proposition \ref{phew}, the right-hand side of the last equality may be written as:
$$\sum \big(g_i^\sigma (a^b)^{[i]}b + g_i' a^{[i]}\big) = \sum \big(g_i^\sigma \sigma(a^{[i]})b +g_i^\sigma (a^{[i]})' + g_i'a^{[i]}\big) = $$ $$\sigma(g(a))b +\sum\big(g_i^\sigma (a^{[i]})' + g_i'a^{[i]}\big).$$ 
By the Leibniz rule, this expression equals: 
$$\sigma(g(a))b +\sum\big(g_i a^{[i]}\big)' = \sigma(g(a))b +g(a)'.$$By the induction hypothesis, this expression equals $\sigma(b^{[n]})b+\delta(b^{[n]}) = b^{[n+1]}$, as needed. \end{proof}

For $a \in D$, $f \in D[x,\sigma,\delta]$ let us define $f^{*n}(a)$ inductively by $f^{*0}(a)  = a$ and $f^{*(n+1)} = f(f^{*n}(a))$. We define the orbit of $a$ by $f$ as the sequence $\big(f^{\circ 0}(a),f^{\circ 1}(a),\ldots \big)$ (where also $f^{\circ 0}(a) = a$). We say that $a$ interchanges with a sequence $A$ (finite or infinite) in $D$ if $a$ interchanges with every element of $A$. 

\begin{theorem}\label{thm1} Let $n \in \N$. If $a$ interchanges with $\big(f^{\circ 0}(a), f^{\circ 1}(a),\ldots,f^{\circ (n-1)}(a)\big)$ or with $\big(f^{*0}(a), f^{*1}(a),\ldots,f^{*(n-1)}(a)\big)$ then $f^{\circ n}(a) = f^{* n}(a)$. \end{theorem}\begin{proof} For $n = 1$  there is nothing to prove. Suppose by induction that we have proven the claim up to $n-1$, and let us prove it for $n$. By the induction hypothesis, the sequence$\big(f^{\circ 1}(a),\ldots,f^{\circ (n-1)}(a)\big)$ coincides with the sequence $\big(f^{* 1}(a),\ldots,f^{* (n-1)}(a)\big)$. In particular, $g = f^{\circ (n-1)}$ satisfies $g(a) = f^{*(n-1)}(a)$. By our assumptions, $a$ interchanges with $g(a)$, hence by Corollary \ref{cor1} $g^{[i]}(a) = \big(g(a)\big)^{[i]}$ for all $i \in \N$. Thus $$f^{\circ n}(a) = \sum f_i g^{[i]}(a) = \sum f_i (g(a))^{[i]} = f(g(a)) = f(f^{*(n-1)}(a)) = f^{*n}(a).$$ \end{proof}

For a given $r \in \N$ we shall say that $a$ is {\bf $r$-periodic} with respect to $f$ if $f^{\circ nr}(a) = a$ for all $n \in \N$. A point $a$ is a fixed point of $f$ if $a$ is $1$-periodic with respect to $f$. 

\begin{theorem}\label{periodic} Let $a \in D, f \in D[x,\sigma,\delta]$ and $r \in \N$.  Suppose that $f^{\circ r}(a) = a$ and that $a$ interchanges with $\big(f^{\circ 0}(a), f^{\circ 1}(a),\ldots,f^{\circ (r-1)}(a)\big)$. Then $a$ is $r$-periodic with respect to $f$ and $a$ interchanges with the whole orbit $\big(f^{\circ 1}(a),f^{\circ 2}(a),\ldots\big)$. \end{theorem} \begin{proof} By Theorem \ref{thm1} $f^{\circ i}(a) = f^{*i}(a)$ for $i = 1,\ldots,r$. In particular, $f^{\circ r}(a) = f^{*r}(a) = a$. It follows that $f^{*(nr+j)}(a) = f^{* j}(a)$ for all $n \in \N$ and $j \in \{0,1,\ldots,r-1\}$, hence $a$ interchanges with $f^{*k}(a)$ for all $k \in \N$. Thus by Theorem \ref{thm1} we have $f^{*k}(a) = f^{\circ k}(a)$ for all $k \in \N$. Thus $a$ interchanges with its orbit $\big(f^{\circ 1}(a),f^{\circ 2}(a),\ldots\big) = \big(f^{* 1}(a),f^{*2}(a),\ldots\big)$, and $f^{\circ nr}(a) = f^{*0}(a) = a$ for all $n \in \N$. \end{proof}

We note that in the special case where $\sigma = \operatorname{Id}$ and $\delta = 0$, Theorem \ref{periodic} coincides with Theorem 4.6 of \cite{CV21}. 

\begin{remark} Given that $f^{\circ r}(a) = a$, the condition that $a$ interchanges with $\big(f^{\circ 0}(a), f^{\circ 1}(a),\ldots,f^{\circ (r-1)}(a)\big)$ is only a sufficient condition for $a$ to be $r$-periodic with respect to $f$, not generally a necessary one. For example, let $R = \HH[x]$, where $\HH$ is the real quaternion algebra with basis $1,i,j,k$ over $\R$.  Let $f = j + k -x \in R$. It is easy to verify by induction that $f^{\circ n}(x) = x$ if $n$ is even, and $f^{\circ n}(x) = f$ if $n$ is odd. In particular, any point $a \in\HH$ is $2$-periodic. However, the point $a = k$, for example, does not commute (that is, does not interchange with) with $f(k) = j$. \end{remark}

Via Theorem \ref{periodic} we also get a different proof of Theorem \ref{thm:fixed}:

\begin{corollary} Let $a \in D, f \in D[x,\sigma,\delta]$. If $f(a) = a$ then $a$ is a fixed point of $f$. \end{corollary} \begin{proof} Since $a$ interchanges with itself we may apply Theorem \ref{periodic} for $r = 1$ to get that $a$ is $1$-periodic with respect to $f$. \end{proof}

\providecommand{\bysame}{\leavevmode\hbox to3em{\hrulefill}\thinspace}
\providecommand{\MR}{\relax\ifhmode\unskip\space\fi MR }
\providecommand{\MRhref}[2]{%
  \href{http://www.ams.org/mathscinet-getitem?mr=#1}{#2}
}
\providecommand{\href}[2]{#2}

\end{document}